\documentclass[12pt,a4paper]{article}

\usepackage{graphicx}
\usepackage{amssymb}
\usepackage{amsmath}
\usepackage{epstopdf}

\usepackage{float}

\usepackage{caption}
\usepackage{subcaption}

\usepackage[margin=0.8in]{geometry} % for the margins (for a 12pt document, the default margin is 1.5 in)

\usepackage{subfloat}

\title{A transform pair for bounded convex planar domains}

\author{J. J. Hulse$^{1}$, L. Lanzani$^{1,2}$, S. G. Llewellyn Smith$^{3,4}$ and E. Luca$^{5}$}

\date{}

\graphicspath{{Fig/}}

\numberwithin{equation}{section}

\usepackage{pst-plot,booktabs,mathtools}
\usepackage{amsfonts, amsthm, xcolor}
\usepackage{graphicx}
\newcommand{\A}{\mathcal{A}}
\newcommand{\B}{\mathcal{B}}

\newcommand{\dd}{\mathrm{d}}
\newcommand{\ee}{\mathrm{e}}
\newcommand{\ci}{\mathrm{i}}
\DeclareMathOperator*{\Arg}{arg}

\DeclareMathOperator*{\Imag}{Im}
\DeclareMathOperator*{\Real}{Re}

\newtheorem{lem}{Lemma}
\newtheorem{cor}{Corollary}
\newtheorem{defi}{Definition}
\newtheorem{thm}{Theorem}

\usepackage{color}

\usepackage{subcaption}

\begin{document}

\maketitle

%\tableofcontents

\begin{center}
$^{1}$Department of Mathematics \\
Syracuse University \\
Syracuse, NY 13244-1150, USA
\end{center}

\begin{center}
$^{2}$Department of Mathematics \\
University of Bologna \\
Bologna, Italy
\end{center}

\begin{center}
$^{3}$Department of Mechanical and Aerospace Engineering \\
Jacobs School of Engineering, UCSD \\
La Jolla, CA 92093-0411, USA.
\end{center}

\begin{center}
$^{4}$Scripps Institution of Oceanography, UCSD \\
La Jolla, CA 92039-0213, USA.
\end{center}

\begin{center}
$^{5}$Climate and Atmosphere Research Center \\
The Cyprus Institute \\
Nicosia, 2121, Cyprus \\
\vskip 0.05truein 
Corresponding author: {\tt e.louca@cyi.ac.cy}
\end{center}

\vskip 0.5truein

\begin{center}
{\bf Abstract}
\end{center}
\noindent
A new transform pair which can be used to solve mixed boundary value problems for Laplace's equation and the complex Helmholtz equation in bounded convex planar domains is presented. This work is an extension of Crowdy (2015, CMFT, 15, 655--687) where new transform techniques were developed for boundary value problems for Laplace's equation in circular domains. The key ingredient of the method is the analysis of the so called global relation which provides a coupling of integral transforms of the given boundary data and of the unknown boundary values.  Three problems which involve mixed boundary conditions are solved in detail, as well as numerically implemented, to illustrate how to apply the new approach.

\vspace{1cm}

\vfill\eject

\section{Introduction}

The Unified Transform Method (UTM) - a method for analysing boundary value problems for linear and integrable nonlinear PDEs - was pioneered in the late '90s by A.S. Fokas \cite{Fokas1997}. From the very beginning, the UTM has attracted a great deal of interest in the applied mathematics community. A multitude of versions of the original method have since been developed, each dealing with a specific family of equations.

For the Laplace, biharmonic, Helmholtz and modified Helmholtz equations in convex polygonal domains, the UTM provides integral representations of the solutions in the complex Fourier plane (Fokas \& Kapaev \cite{FokasKapaev2003}, Crowdy \& Fokas \cite{CrowdyFokas2004}, Dimakos \& Fokas \cite{DimakosFokas2015}, Spence \& Fokas \cite{SpenceFokas2010}, Davis \& Fornberg \cite{DavisFornberg2014}). Specifically, for Laplace's equation, Fokas \& Kapaev \cite{FokasKapaev2003} developed a transform method for solving boundary value problems in simply connected polygonal domains. Their original approach relied on a variety of tools (spectral analysis of a parameter-dependent ODE; Riemann--Hilbert techniques, etc.). It was later observed by Crowdy \cite{Crowdy:2015} that the method can be recast within a complex function-theoretic framework; this, in turn, lead to the development of a new transform method applicable to so-called circular domains (domains bounded by arcs of circles, with line segments being a special case). Colbrook \cite{Colbrook:2020} extended the unified transform method to curvilinear polygons and PDEs with variable coefficients. In addition, Colbrook {\em et al.} \cite{ColbrookFokasHash2019} presented a hybrid analytical-numerical technique for elliptic PDEs based on the UTM, providing a fast and efficient method to evaluate the solution in the interior domain.

The focus of the present study is the extension of the original approach of Fokas \& Kapaev \cite{FokasKapaev2003} for convex polygons, to arbitrary convex domains. The method is built upon Crowdy's \cite{Crowdy:2015} construction and it develops a new transform method for any convex bounded domain; this includes domains that may be non-circular or non-polygonal, such as ellipses. This study was motivated by engineering applications, in particular heat exchangers (namely the shell-and-tube exchangers) which have elliptical cross section (Saunders \cite{Saunders1988}) and the need of mathematical tools and transform methods to analyse problems in such geometries. 

In \S 2, we present the theoretical framework needed to formulate the new transform pair for analytic functions in bounded convex planar domains. The new transform pair is presented in \S 3. The next step involves implementing the new transform in a variety of mixed boundary value problems (\S 4--6). In \S 7, we present the formulation of a transform pair for the complex Helmholtz equation. Finally, we conclude and discuss further applications in \S 8.

\section{Theoretical framework}\label{sec2}
\subsection{Preliminaries} 

%In the literature there seems to be no universal agreement on the definition of a trapezoid \cite{Joseffson:2013}

A {\em trapezoid} is defined here to be a quadrilateral with at least one pair of parallel sides (there exist various definitions \cite{Joseffson:2013}). It follows from this definition that a trapezoid is a convex domain, given  that the sum of the interior angles of a trapezoid is equal to $2\pi$.
%Hence a rectangle is a trapezoid.

\begin{lem}\label{L:2}
Let $\Omega\subset \mathbb{C}$ be a bounded convex domain.  For any $z_0 \in \Omega$, there exists a trapezoid $T=T(z_0)$ with the following properties:
\begin{enumerate}
\item Point $z_0$ is contained in the interior of $T$.
\item Two (parallel) sides of $T$ are each parallel to a coordinate axis.
\item The vertices of $T$ lie on the boundary of $\Omega$.
\item The closure of $T$ is contained in the closure of $\Omega$.
\end{enumerate}
%either the $x$-axis or $y$-axis.
\end{lem}

\begin{proof} First, note that a trapezoid satisfying properties 1--4 is not uniquely determined. To construct one such trapezoid, we first find an open disk lying in $\Omega$ containing $z_0$. In this disk, we inscribe a rectangle that contains $z_0$ with sides parallel to the coordinate axes. Ignoring, say, the vertical sides, we extend the horizontal sides of the rectangle until their vertices reach the boundary of $\Omega$. There is now only one way to connect these new vertices with straight line segments as to form a trapezoid that satisfies properties 1--4. Finally, the convexity assumption on $\Omega$ ensures that the closure of $T$ lies in the closure of $\Omega$.
\end{proof}

\noindent
{\it Remark 1:} Without loss of generality, we assume
%once and for all
that a coordinate system has been fixed where the coordinate axis of property 2 in Lemma \ref{L:2} is the $x$-axis. %(the horizontal axis).

\vskip0.05in

\noindent
{\it Remark 2:} The vertices of $T$ partition the boundary of $\Omega$ into four adjacent arcs, where each arc is subtended by precisely one side of $T$.

\begin{lem}%\label{L:Cauchy}
\label{E:limit}
For any $z,\zeta \in \mathbb{C}$ such that $0<\Arg(z-\zeta)<\pi$, we have
\begin{equation}\label{E:limit2}
\lim_{t\rightarrow +\infty }\ee^{\ci t(z-\zeta)}=0.
\end{equation}
Here $\ci$ is the imaginary unit. Hence
\begin{equation}\label{E:kernel}
 \frac{1}{\zeta-z}\ =\ci \int\limits_0^\infty \ee^{\ci t(z-\zeta)}\dd t.
 \end{equation}
\end{lem}

\begin{proof}
The proof is a computation.
\end{proof}

%% Proof of the lemma:
%\begin{proof} Writing $z-\zeta=r(\cos\theta+\ci \sin\theta)$ where $\theta=\Arg(z-\zeta)$, we see that
%$$
%\left|\ee^{\ci t(z-\zeta)}\right|\ = \ee^{-t\,r\sin\theta}\, .
%$$
%Now for $t>0$, the exponent on the righthand side is negative by the hypothesis that $0<\theta<\pi$, thus granting \eqref{E:limit} which, in turn, gives \eqref{E:kernel} once we recall that
%\begin{align*}
%\ci \int_0^\infty \ee^{\ci t(z-\zeta)}\dd t=\lim_{t\rightarrow \infty} \frac{1}{z-\zeta}\big[\ee^{\ci t(z-\zeta)}-1\big].
%\end{align*}
%by the fundamental theorem of calculus.
%\end{proof}

 \subsection{A labelling scheme} 
 
For domain $\Omega$ and trapezoid $T$ introduced in Lemma \ref{L:2}, we let $T_k$ and $v_k$, $k=1,\ldots, 4$ denote the sides and vertices of $T$, respectively, where each label $k$ is assigned as follows. We label the lowest horizontal side $T_1$, and then label the
remaining sides in ascending order as we travel along boundary of $T$ in the counterclockwise direction starting from $T_1$. For the vertices of $T$, we set 
$v_1$ to be the leftmost vertex of the side $T_1$ and then label the remaining vertices in ascending order as we travel along the boundary of $T$ in the counterclockwise direction. 
Finally, the four arcs that partition the boundary of $\Omega$ (see Remark 2) are denoted $I_k$, $k=1,\ldots, 4$, where $I_k$ is subtended by $T_k$; see Figures \ref{Fig:1}--\ref{Fig:2}. \\

\noindent
{\it Remark 3:} Since $\Omega$ is convex, we have that $I_k$ and the interior of $T$ lie on opposite sides of the line determined by $T_k$ for each $k=1,\ldots, 4$.

\begin{figure*}[t!]
    \centering
    \begin{subfigure}[t]{0.5\textwidth}
        \centering
        \[\includegraphics[scale=0.8] {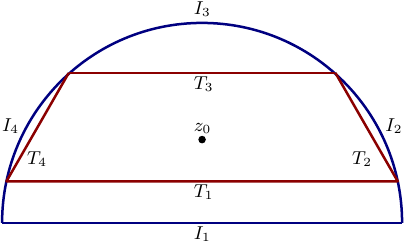} \]
\caption{{\em A trapezoid $T$ (in red) inscribed in $\Omega$ (in blue), with labelling scheme.}}
\label{Fig:1}
    \end{subfigure}%
    ~ 
    \begin{subfigure}[t]{0.5\textwidth}
        \centering
  \[\includegraphics[scale=0.8] {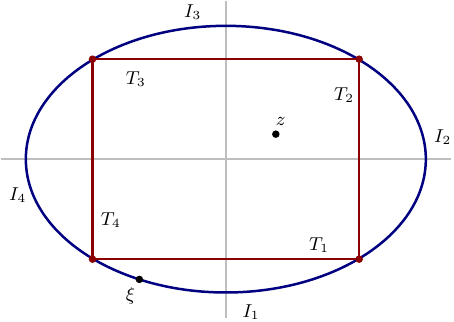} \]
  \caption{{\em A trapezoid $T$ (in red) inscribed in another $\Omega$ (in blue), with labelling scheme.}}
  \label{Fig:2}
    \end{subfigure}
    \caption{Example domains $T$ and $\Omega$.}
\end{figure*}

\begin{lem}\label{L:3}
Let $\Omega\subset\mathbb C$ be a given bounded convex domain; let $T$ be a trapezoid as in Lemma \ref{L:2}, and let $\{I_1,\ldots, I_4\}$ be the corresponding partition of the boundary of $\Omega$. Then, for any $\alpha\in T$ and for any $k=1,\ldots, 4$ there is $\beta_k\in [0, 2\pi)$
    %$\beta_k\in\mathbb R$ 
such that the conformal affine map 
\begin{equation}\label{E:Psi}
\Psi_k (w) \coloneqq \ee^{-\ci \beta_k}(w-\alpha),\quad w \in \mathbb{C},
\end{equation}
  %  rectangle with sides $R_k$, $k=1,\ldots, 4$, if $(z,\zeta)\in \textrm{int}(R)\times I_k$, $1\leq k\leq 4$, then 
has
\begin{equation}\label{E:in1}
\Imag{(\Psi_k(\zeta))}<0,
\end{equation}
and
\begin{equation}\label{E:in2}
\Imag(\Psi_k(\zeta))<\Imag(\Psi_k(z)),
\end{equation}
for any  $\zeta \in I_k$ and for any $z\in \textrm{int}(T)$.
\end{lem}

 \begin{proof}
Since $\alpha$ is an interior point of $T$, it has positive distance from the boundary of $T$, including $T_k$, that is
\begin{equation}\label{E:L1}
0< \text{dist}(\alpha, T_k) \coloneqq \inf\{|\alpha-z|: z\in T_k\}.
\end{equation}
We now choose $\beta_k\in \mathbb{R}$ so that $\Psi_k$ in \eqref{E:Psi} has
\begin{equation}\label{E:L2}
\Psi_k(T_k)\subset \{z: \Imag\, z = -\text{dist}(\alpha, T_k)\}.
\end{equation}
But the interior of $T$ and $I_k$ lie on opposite sides of $T_k$ (Remark 3.) and affine maps are rigid motions, thus $\Psi_k(int (T))$ and $\Psi_k(I_k)$ must lie on opposite sides of $\Psi_k(T_k)$ as well, and we choose $\beta_k\in [0, 2\pi)$
    % $\beta_k\in\mathbb R$
so that 
  %  we may choose $\beta_k\in\mathbb R$ so that
\begin{equation}\label{E:L3}
\Psi_k(I_k)\subset \{w:  \Imag\, w\leq -\text{dist}(\alpha, T_k)\},
\end{equation}
whereas
\begin{equation}\label{E:L4}
\Psi_k(int(T)) \subset \{w: \Imag\, w> -\text{dist}(\alpha, T_k)\}.
\end{equation}
Now \eqref{E:in1} follows from \eqref{E:L1} and \eqref{E:L3}, while \eqref{E:L1} and \eqref{E:L4} give \eqref{E:in2}.
%\Red{[SGLS: Equation number (3.14) seem off?]}
 \end{proof}

 \begin{cor}\label{C:argument}
 With same notations and hypotheses as above, we have that
 \begin{equation}\label{E:argument}
 0<\Arg\big(\Psi_k(z)-\Psi_k(\zeta)\big)<\pi,
 \end{equation}
 for any $z\in int (T)$ and any $\zeta\in I_k$.
 \end{cor}
 This corollary is just a reformulation of $\eqref{E:in2}$ in such a way that it highlights the geometric condition needed to apply Lemma 2.
% 
% \begin{proof}
%By \eqref{E:in2} in Lemma 3, we have that $\Imag\big(\Psi_k(z)-\Psi_k(\zeta)\big)>0$ for any $z\in\text{int}(T)$ and any $\zeta\in I_k$. Hence the imaginary part of $\Psi_k(z)-\Psi_k(\zeta)$ lies in the upper half plane, thus giving
% \eqref{E:argument}.
% \end{proof}

 \section{A new transform pair}
 
With same notations as before, we may now give the following definitions. Let $\Omega$ be a simply connected and bounded domain. Let $C_1, C_2,\dots$ be a sequence of rectifiable Jordan curves that converges to $\partial \Omega$ in the sense that each compact subdomain of $\Omega$ is eventually contained in $C_n$ for large enough $n$. For $0 < p\leq \infty$, we say that a holomorphic function $f\in O(\Omega)$ is in the Hardy Space $E^p(\Omega)$ (also known as Smirnov Class), if there is a positive and finite constant $M=M(f)$ such that
\begin{equation}
\int_{C_n} |f(\zeta)|^p \dd \sigma<M<\infty,
\end{equation}
for all $n$ and where $\dd \sigma$ is arclength. If we further assume that  $\Omega$ is bounded by a rectifiable Jordan curve, then we have that $f$ has a nontangential limit on $\partial\Omega$ $d\sigma$ almost everywhere which we also denote by $f$ and
\begin{equation}
\int_{\partial\Omega}|f|^p \dd \sigma(\zeta)<\infty.
\end{equation}
Note that $E^p(\Omega)$ will contain $O(\Omega)\cap C(\overline{\Omega})$. We have the following classical result found on page 170 of Duren's \textit{Theory of $H^p$ Spaces} \cite{Duren:1970}.

\begin{thm}
If $f\in E^1(\Omega)$, then
\begin{equation}
f(z)=\frac{1}{2\pi i}\int_{\partial\Omega}\frac{f(\zeta)}{\zeta-z}\dd \zeta, \qquad z \in \Omega,
\end{equation}
and the integral vanishes for all $z$ outside of $\partial\Omega$.

Conversely, if $g\in L^1(\partial \Omega)$ and 
\begin{equation}
\int_{\partial \Omega}\zeta^ng(\zeta)\dd \zeta=0, \qquad n=0,1,2,\dots,
\end{equation}
then
\begin{equation}
f(z)=\frac{1}{2\pi \ci}\int_{\partial\Omega}\frac{g(\zeta)}{\zeta-z} \dd \zeta\in E^1(\Omega),
\end{equation}
and $g$ coincides $\dd \sigma$ almost everywhere on $\partial \Omega$ with the nontangential limit of $f$.
\end{thm}

Recall that $\dd \sigma=d|\zeta|$ denotes arclength. In particular the theorem above implies that if $f\in E^1(\Omega)$ then $\int_{\partial \Omega}f(\zeta)\dd \zeta=0$. Note that a detailed discussion on Hardy spaces over general domains is given in \cite{Duren:1970}. In the subsequent work, we will be interested in the Hardy space on bounded and convex domains. Convexity of a compact set in $\mathbb{R}^n$ implies the boundary is rectifiable. Thus we need not mention that the boundary of $\Omega$ must be rectifiable. 
%See Chapter 10 of \cite{Duren:1970} for a detailed discussion of Hardy Spaces over general Domains.

\begin{defi}\label{D:definition_spectral-jk}
Let $\Omega\subset\mathbb C$ be a bounded convex domain bounded by a Jordan curve and let $f\in E^1(\Omega)$. The \textbf{spectral matrix $\{\rho_{jk}(t) : t\in\mathbb C\}_{jk}$} has components 
\begin{equation}\label{D:spectral-jk}
\displaystyle{\rho_{jk}(t)=\int\limits_{I_k}f(\zeta)\ee^{-\ci t \ee^{{-\ci \beta_j}}\zeta} \dd \zeta},\quad t\in\mathbb C, \quad j, k=1,\ldots, 4,
\end{equation}
with $I_k$ and $\beta_j$ defined in \S\ref{sec2}. We shall refer to functions $\rho_{jk}(t)$ as \textbf{spectral functions}.
\end{defi}

\begin{lem}\label{L:4}
Using the same notation and assumptions as in Definition \ref{D:definition_spectral-jk} above, the spectral functions $\rho_{jk}(t)$ are entire function for all $j, k=1,\ldots, 4$, and 
\begin{equation}\label{E:global-relation}
\sum_{k=1}^4\rho_{jk}(t)=0, \quad t \in \mathbb{C},\quad j=1,\ldots, 4.
\end{equation}
We refer to \eqref{E:global-relation} as the \textbf{global relations}.
\end{lem}

 \begin{thm}     
%[H-L-LS-L, 2023]
Let $\Omega\subset\mathbb C$ be a bounded convex domain bounded by a Jordan curve, let $T$ be a trapezoid as in Lemma \ref{L:2}, and let $f\in E^1(\Omega)$. Then, using the same notations as above, we have that
\begin{equation}\label{E:repr-f}
f(z)=\frac{1}{2\pi }\sum_{j=1}^4 \int\limits_0^\infty \rho_{jj}(t) \, \ee^{-\ci \beta_j} \ee^{\ci t \ee^{-\ci \beta_j}z} \dd t, \qquad \text{for any } z \in int (T).
\end{equation}
We refer to \eqref{D:spectral-jk} and \eqref{E:repr-f} as the \textbf{transform pair} for bounded convex domains.
%  Furthermore
%Let $T$ be a trapezoid whose closure is contained in the closure of $\Omega$ and has vertices on the boundary of $\Omega$. Then for any $z\in int (T)$,
%\[ \begin{cases}\displaystyle{f(z)=\frac{1}{2\pi }\sum_{j=1}^4 \int\limits_0^\infty e^{-i\beta_j}e^{ite^{-i\beta_j}z}
%    \rho_{jj}(t) dt}\\ 
%    \\
%    \displaystyle{\rho_{jk}(t)=\int\limits_{I_k}f(\zeta)e^{-ite^{-i\beta_j}\zeta}d\zeta}\end{cases}\]
%    where $\beta_k$ are four distinct real numbers that depend on the trapezoid $T$ and $I_k$, $1\leq j\leq 4$, are the segments of the boundary of $\Omega$ partitioned by the vertices of $T$. The pair above is called the transform pair. For any $t\in \mathbb C$, the following global relation also holds:
\end{thm}

\begin{proof} 
Fix $\alpha\in T$ and recall the conformal affine  maps obtained in Lemma \ref{L:3}:
\begin{equation}
\ee^{-\ci \beta_j}(w-\alpha) = \Psi_j(w),\quad  1\leq j\leq 4\, ,\quad w\in\mathbb C.
\end{equation}
%see \eqref{E:Psi},
Hence
\begin{equation}\label{E:star}
\ee^{-\ci \beta_j}(z-\zeta) = \Psi_j(z)-\Psi_j(\zeta), \quad \text{for any}\ z, \zeta\in\mathbb C.
\end{equation}
It follows from \eqref{E:star} and the definition of $\rho_{jj}(t)$ that
%see \eqref{D:spectral-jk}
\begin{equation}\label{E:star-star-star}
\ee^{-\ci \beta_j} \ee^{\ci t \ee^{-\ci \beta_j}z} \rho_{jj}(t) = \int\limits_{I_j}\!\!f(\zeta)\, \ee^{-\ci \beta_j} \ee^{\ci t (\Psi_j(z)-\Psi_j(\zeta) )}
  % e^{-ite^{-i\beta_j}\zeta}
\dd \zeta.
 %} 
\end{equation}
%if $(z,\zeta)\in int(T)\times I_j$, $1\leq j\leq 4$. 
This can be written as
\begin{equation}\label{E:aux1}
\ee^{-\ci \beta_j}\ee^{\ci t \ee^{-\ci \beta_j}z} \rho_{jj}(t) = \int\limits_{I_j}\!\!f(\zeta)\, h_{z, \zeta}(t)\, \dd\zeta,
%\label{E:aux1}
\end{equation}
where
\begin{equation}
h_{z, \zeta}(t) \coloneqq \ee^{-\ci \beta_j}\ee^{\displaystyle{\ci t(\Psi_j(z)-\Psi_j(\zeta))}},\quad t\in\mathbb{R}^{+}.
% \in L^1(\mathbb R^+, dt).
\end{equation}

%we have by
%% Lemma \ref{L:1} 
%Corollary \ref{C:argument} that 
%%$\Imag{\Psi_k(\zeta)}<\Imag{\Psi_k(z)}$ and $\Imag{\Psi_k(\zeta)}<0$. Thus 
%$0<\textrm{Arg}(\Psi_j(z)-\Psi_j(\zeta))<\pi$.
%Hence it follows by 
%By Lemma \ref{E:limit} that
%%Proposition \ref{P:1},
%\begin{equation}\label{E:aux2}
%\frac{1}{\Psi_j(\zeta)-\Psi_j(z)}=i\int\limits_0^\infty e^{\displaystyle{it\Big(\Psi_j(z)-\Psi_j(\zeta)\Big)}}dt
%%\quad \text{for any}\ (z,\zeta)\in \textrm{int}(T)\times I_k
%\end{equation}
%for any $(z,\zeta)\in \textrm{int}(T)\times I_j$. 
%\vskip0.1in

\noindent
We claim that, for each $(z,\zeta)$, the absolute value of the function $h_{z,\zeta}$ (as a function of $t$) has finite integral on $\mathbb{R}^{+}$; more precisely, 
\begin{equation}%\label{E:L1}
h_{z, \zeta}(t) \in L^1(\mathbb{R}^{+}, \dd t), \quad \text{for any}\ (z, \zeta)\in int(T)\times I_j.
\end{equation}
To see this, note that
\begin{equation}
|h_{z, \zeta}(t)| = \ee^{\displaystyle{-t\,\Imag (\Psi_j(z)-\Psi_j(\zeta))}},\quad t \in \mathbb{R}^+.
\end{equation}
But \eqref{E:in2} implies that
\begin{equation}
c_{j, \zeta} \coloneqq \Imag (\Psi_j(z)-\Psi_j(\zeta)) > 0.
\end{equation}
Hence
\begin{equation}
\int\limits_0^\infty |h_{z, \zeta}(t)|\, \dd t =
\int\limits_0^\infty \ee^{\displaystyle{-t\,c_{j, \zeta}}}\,\dd t=\lim_{t\to \infty}\frac{1}{c_{j,\zeta}}\left[1-\ee^{\displaystyle{-t\,c_{j, \zeta}}}\right] 
%= -\frac{1}{c} 
= \frac{1}{c_{j, \zeta}}<\infty,
\end{equation}
thus proving the claim.  Note that there exists some $\epsilon>0$ such that $c_{j,\zeta}>\epsilon$ for $\zeta\in I_j$ (this follows from \eqref{E:L3} and \eqref{E:L4}).

Next we observe that since $f\in E^1(\Omega)$  then in particular $f \in L^1(\partial\Omega, \dd \sigma)$ and hence $f \in L^1(I_j, \dd\sigma)$, $j=1,\ldots, 4$, where $\dd \sigma = |\dd \zeta|$ is the arclength measure for $\partial \Omega$.

Using Cauchy's integral formula 
%and Fubini's theorem (see \eqref{E:L1} and \eqref{E:L1bis}),
\begin{equation}
f(z)=\frac{1}{2\pi \ci} \int\limits_{\partial \Omega}\frac{f(\zeta)}{\zeta-z}\dd \zeta=\frac{1}{2\pi \ci} \sum_{j=1}^4\int\limits_{I_j}\frac{f(\zeta)}{\zeta-z}\dd \zeta,\qquad \text{for any}\ z \in \Omega.
\end{equation}
But
\begin{equation}\label{E:star-star}
\frac{1}{\zeta-z} = \frac{\ee^{-i\beta_j}}{\Psi_j(\zeta)-\Psi_j(z)}, \quad \text{for any}\ z\neq\zeta,
\end{equation}
so that
\begin{equation}
f(z)=\frac{1}{2\pi \ci} \int\limits_{\partial \Omega} \frac{f(\zeta)}{\zeta-z}\,\dd \zeta=\frac{1}{2\pi \ci} \sum_{j=1}^4 \int\limits_{I_j} \frac{f(\zeta)\ee^{-\ci \beta_j}}{\Psi_j(\zeta)-\Psi_j(z)}\dd\zeta, \qquad \text{for any}\ z\in \Omega.
\end{equation}
%Now, Corollary \ref{C:argument} gives that
Now, by Corollary \ref{C:argument}, we have that
\begin{equation}
0<\textrm{Arg}(\Psi_j(z)-\Psi_j(\zeta))<\pi\quad \text{whenever}\ \ (z, \zeta)\in \text{int}(T)\times I_j.
\end{equation}
Hence we may apply Lemma \ref{E:limit} and invoke the definition of $h_{z, \zeta}(t)$ to conclude that
\begin{equation}\label{E:aux2}
\frac{\ee^{-\ci \beta_j}}{\Psi_j(\zeta)-\Psi_j(z)}=\ci \int\limits_0^\infty h_{z, \zeta}(t)\, \dd t,
%e^{\displaystyle{it\Big(\Psi_j(z)-\Psi_j(\zeta)\Big)}}dt
%\quad \text{for any}\ (z,\zeta)\in \textrm{int}(T)\times I_k
\end{equation}
for any $(z,\zeta)\in int(T)\times I_j$. Combining all of the above, we obtain
\begin{equation}\label{f3.23}
f(z) = \frac{1}{2\pi}\sum_{j=1}^4\int\limits_{I_j} f(\zeta) \int\limits_0^\infty h_{z, \zeta}(t) \dd t \dd\zeta.
\end{equation}
%\Red{[SGLS: explain why these conditions are needed to apply Fubini's theorem.]}
Again, note that there exists some $\epsilon>0$ such that $c_{j,\zeta}>\epsilon$ for $\zeta\in I_j$ (this follows from \eqref{E:L3} and \eqref{E:L4}). Thus the integrand in \eqref{f3.23} is bounded above uniformly in $\zeta \in I_{j}$,  hence we may apply Fubini's theorem to exchange the order of integration and obtain
\begin{equation}
f(z) = \frac{1}{2\pi}\sum_{j=1}^4 \int\limits_0^\infty\int\limits_{I_j} f(\zeta) h_{z, \zeta}(t) \dd \zeta \dd t,
\end{equation}
but by the definition of $h_{z, \zeta}(t)$ (see \eqref{E:star-star-star} and \eqref{E:aux1}), the latter is precisely the righthand side of \eqref{E:repr-f}, completing the proof of the theorem.
\end{proof} 

\noindent  
{\bf Remark 4.} 
   %The role of the trapezoid $T$ is confined to the fact that the vertices of $T$ partition the boundary of $\Omega$, and that $z$ must be in the interior of $T$. Hence, 
Given $z\in \Omega$ there are many different choices of trapezoids $T$
% $T\ni z$
that will work, each giving rise to a transform pair. In fact, any convex polygon inscribed in $\Omega$ will give rise to a transform pair. Each convex polygon will result in a unique partition of the boundary of $\Omega$. All such partitions (one for each choice of convex polygon) will produce comparable outcomes. 

\section{Dual Fourier Series for the disc}\label{SS:DFFSdisc}

We consider the following {\em incomplete Dirichlet boundary value problem for analytic functions in the unit disc $\mathbb{D} \coloneqq \{z \in \mathbb{C} ~| ~ |z|<1\}$} for a given $m\in\mathbb N$:
\begin{equation}\label{E:BVP-circ}
\begin{cases}\overline{\partial}f (z) = 0, \qquad |z|<1, \\ \\
\Real f(\ee^{\ci \theta})=\cos m\theta, \qquad \theta \in C_1 \coloneqq (-\pi/2, \pi/2), \\ \\
\Imag f(\ee^{\ci \theta})=-\sin m\theta, \qquad \theta \in C_2 \coloneqq  (\pi/2, 3\pi/2).
\end{cases}
\end{equation}
%\vskip0.2in

\noindent
(We refer to this problem as ``incomplete", because we only provide just the real part or the imaginary part of the boundary data, on each piece of the boundary.)  
This problem is uniquely solvable; it was originally solved by Shepherd \cite{Shepherd1937} via a sequence of ingenious manipulations of integral representations of the Fourier series and was revisited by Crowdy \cite{Crowdy:2015}.

Here we proceed as follows to obtain the solution. In the first step, we use one of the global relations \eqref{E:global-relation}  to extend the given (incomplete) boundary data to a function $f(\zeta)$ defined on the full boundary $\partial \mathbb{D}$.
%   : it is important to point out that, in general, the global relation will not provide a closed formula for $f(\zeta)$ but rather its numerical approximation. 
In the second step, we use our transform pair to numerically compute at a given point $z \in \mathbb{D}$ the (unique) solution of the completion of \eqref{E:BVP-circ} to a Dirichlet problem on $\partial\mathbb{D}$. 

We point out that the Dirichet problem above will have a unique solution in $E^\infty(\mathbb{D})$. Such a solution cannot be expressed in closed form, but can be computed numerically by the two steps mentioned previously. Note that the boundary data above may have many other solutions that to not belong to $E^p(\mathbb{D})$ for any $p\geq 1$.
%   \vskip0.1in
%   We point out that the boundary data given in \eqref{E:BVP-circ} has many extensions to the full boundary, but  only one such extension will lead to the solution of \eqref{E:BVP-circ}! 
For instance, it is immediate to see that
\begin{equation}\label{E:obvious}
f(\ee^{\ci \theta}) \coloneqq \ee^{-\ci\,m\theta},\quad \theta\in (-\pi/2, 3\pi/2],
\end{equation}
is an extension of the given data to the full boundary, in fact: 
\begin{equation}\label{E:obvious-circle}
f(\zeta) = \overline{\zeta}^{\,m},\quad \zeta\in\partial\mathbb D.
\end{equation}
However this choice of $f(\zeta)$ has the unique harmonic extension 
%\Red{[SGLS: explain in what sense this is unique - in what class of functions?]}
\begin{equation}
f(z) \coloneqq \overline{z}^{\,m}, \quad z\in \overline{\mathbb{D}},
\end{equation}
which does not solve \eqref{E:BVP-circ} because it is not holomorphic. The main thrust of the global relations \eqref{E:global-relation} is that any of them drives the numerical construction of the unique extension
 %of the given data 
to $\partial \mathbb D$ belonging to $E^\infty(\mathbb{D})$ that solves \eqref{E:BVP-circ}.

The trapezoid in the transform pair derived above may be taken to be a rectangle with sides parallel to the coordinate axis. Thus for any $z\in\mathbb{D}$, we may find a rectangle $R$ as described in the previous sentence with $z\in R$ and use the following {\em rectangle-like} transform pair:
\begin{equation}\label{E:TP}
\begin{cases}
\displaystyle{f(z)=\frac{1}{2\pi }\sum\limits_{n=1}^4 \int\limits_{\mathcal L} c_n \ee^{\ci c_ntz} \rho_{nn}(t) \dd t},\quad |z|<1, \\ \\
\displaystyle{\rho_{jn}(t)=\int\limits_{I_n(z)}\!\!\!f(\zeta)\ee^{-\ci c_jt\zeta}\dd\zeta,\quad n, j=1, \ldots, 4},
\end{cases}
\end{equation}
where
\begin{itemize}
\item $c_j \coloneqq \ee^{\ci \beta_j}$, $j=1,\ldots 4$ where  $\beta_j=\pi(j-1)/2$, $n=1,2,3,4$.
\item $\mathcal L$ is the {\em fundamental contour}. $\mathcal L \coloneqq (0, +\infty)$ (hence $t\in\mathbb R^+$)
\item $\{I_n(z), \ n=1,\ldots, 4\}$ is a partition of the unit circle into four disjoint ``coordinate arcs'' (each subtended by two points that determine a horizontal or vertical straight line segment) determined by the evaluation point $z$ in $\mathbb{D}$. 
\end{itemize}

On $C_1$, we write, as in \cite{Crowdy:2015},
\begin{equation}\label{F:one}
f(\ee^{\ci \theta})=\cos m\theta + \ci \Big(a_0+\sum_{n\geq 1}\Big[a_n \ee^{2\ci n\theta}+\overline{a_n} \ee^{-2\ci n\theta}\Big]\Big), \quad \theta\in (-\pi/2, \pi/2),
\end{equation}
where the coefficients $a_0 \in \mathbb{R}$ and $\{a_n\in\mathbb C|n=1,2,\dots \}$ are to be determined. Similarly, on $C_2$, we write
 %we can write the real part of $f$ as a period- $\pi$ {\color{red} Taylor} series over the interval {\color{red} $(-\pi/2,\pi/2)$}
\begin{equation}\label{F:two}
f(\ee^{i\theta})= b_0+\sum_{n\geq 1}\Big[b_n \ee^{2 \ci n\theta}+\overline{b_n} \ee^{-2\ci n\theta}\Big]-\ci \sin m\theta\, ,\quad \theta \in (\pi/2, 3\pi/2),
\end{equation}
where coefficients $b_0\in \mathbb{R}$ and $\{b_n\in\mathbb C|n=1,2,\dots \}$ are to be determined.  Recall the global relation given in \eqref{E:global-relation}: %and its proof
\begin{equation}\label{E:global}
\sum_{n=1}^4\rho_{1n}(t)=\sum_{n=1}^4\,\int\limits_{I_n(z)}\!\!\! f(\zeta)\ee^{-\ci t\zeta}\dd \zeta=\int\limits_{\partial \mathbb D}f(\zeta)\ee^{-\ci t\zeta}\dd\zeta =0, \quad t\in \mathbb{C}.
\end{equation}

%By substituting  the Fourier series representations of $f$ on $C_1$ and $C_2$ into \eqref{E:global},
%the equation above {\color{red}\tt[ which equation?]}, %{\color{red} \tt [Question: how are $C_1$ and $C_2$ related to $I_n$, $n=1,\ldots, 4$?]} 
%for each fixed $t\in\mathbb C$ we obtain a linear equation with an infinite number of unknowns (namely, the Fourier coefficients of $f$).
%   By making different choices of $t$ over the complex field $\mathbb C$, we generate as many linear equations as needed and we thus create a system whose solution are the first $N$ Fourier coefficients of $f(z)$.

On substitution of \eqref{F:one} and \eqref{F:two} into \eqref{E:global}, we obtain a linear system for the unknown coefficients given by
%\begin{equation}
%\begin{split}
%        %=\int_{\partial \mathbb D} f(\zeta)e^{-it\zeta}d\zeta=\int_{C_1}f(\zeta)e^{-it\zeta}d\zeta+\int_{C_2}f(\zeta)e^{-it\zeta}d\zeta\\&
%\int\limits_{-\pi/2}^{\pi/2}&\Big[\cos m\theta +\ci[a_0+\sum_{n\geq 1}a_n \ee^{2n \theta \ci}+\overline a_n \ee^{-2n \theta \ci}]\Big]\ee^{-\ci t \ee^{i\theta}} \ci \ee^{\ci \theta}\dd\theta \\
%&+\int\limits_{\pi/2}^{3\pi/2}\Big[b_0+\sum_{n\geq 1}b_n \ee^{2n\theta \ci}+\overline b_n \ee^{-2n\theta \ci} - \ci \sin m\theta \Big] \ee^{-\ci t \ee^{\ci \theta}}\ci \ee^{\ci \theta}\dd\theta=0.
%\end{split}
%\end{equation}
%Rearranging terms give          
\begin{equation}\label{S:one} %\label{S:two}
\begin{split}
&a_0\A(0,t)+\sum_{n\geq 1}\Big(a_n\A(2n,t)+\overline{a_n} \A(-2n,t)\Big) \\
&~~~~~~+b_0\B(0,t)+\sum_{n\geq 1}\Big(b_n \B(2n,t)+\overline{b_n} \B(-2n,t)\Big)=r(t), \quad t \in \mathbb{C}, \\
&a_0\overline{\A(0,t)}+\sum_{n\geq 1}\Big(\overline{a_n} \overline{\A(2n,t)}+ a_n\overline{\A(-2n,t)}\Big) \\
&~~~~~~+b_0\overline{\B(0,t)} +\sum_{n\geq 1}\Big(\overline{b_n} \overline{\B(2n,t)}+ b_n\overline{\B(-2n,t)}\Big)=\overline{r(t)}, \quad t \in \mathbb{C},
\end{split}
\end{equation}
where
\begin{equation}\label{ABdisc}
\begin{split}
\A(n,t)&=- \int_{-\pi/2}^{\pi/2}\ee^{-\ci t \ee^{\ci\theta}}\ee^{\ci (n+1) \theta}\dd\theta, \\ 
\B(n,t)&=\ci \int_{\pi/2}^{3\pi/2}\ee^{-\ci t\ee^{\ci\theta}}\ee^{\ci (n+1) \theta} \dd\theta,
\end{split}
\end{equation}
and the function $r(t)$ is defined by
\begin{equation}\label{rdisc}
r(t)=-\int_{\pi/2}^{3\pi/2} \sin(m\theta)\ee^{-\ci t\ee^{\ci \theta}} \ee^{\ci \theta}\dd\theta - \ci \int_{-\pi/2}^{\pi/2}\cos(m\theta)\ee^{-\ci t \ee^{\ci\theta}} \ee^{\ci \theta}\dd \theta.
\end{equation}
%\begin{equation}
%\begin{split}
%\A(n,t)&=\ci \int_{-\pi/2}^{\pi/2}\ee^{-\ci t \ee^{\ci\theta}}\ee^{(n+1)\ci\theta}\ci\dd\theta, \\ 
%\B(n,t)&=\int_{\pi/2}^{3\pi/2}\ee^{-\ci t\ee^{\ci\theta}}\ee^{(n+1)\ci \theta}\ci \dd\theta, \\
%r(t)&=\int_{\pi/2}^{3\pi/2}\ci\sin(m\theta)\ee^{-\ci t\ee^{\ci \theta}}\ci \ee^{\ci \theta}\dd\theta-\int_{-\pi/2}^{\pi/2}\cos(m\theta)\ee^{-\ci t \ee^{\ci\theta}}\ci \ee^{\ci \theta}\dd \theta.
%\end{split}
%\end{equation} 
The equations above hold for all $t \in \mathbb{C}$. To proceed, the sums  \eqref{F:one} and  \eqref{F:two} (for $f(e^{i\theta})$) are truncated to include only terms up to $n=N$, and we formulate a linear system for the unknown coefficients $\{a_n, b_n | n=0,\dots,N \}$. The linear system comprises of conditions \eqref{S:one} evaluated at points $t \in \mathbb{C}$ which are used to form an overdetermined linear system. This is then solved using least squares. Following an approach similar to Colbrook, {\em et al.} \cite{ColbrookFlyerFornberg:2018}, points $t \in \mathbb{C}$ are chosen to be the origin and the concentric circles
\begin{equation}
\{t \in \mathbb{C} | t=-j/f'(\ee^{\ci \theta}), \theta \in [0,2\pi], j=1,\dots,M_r\},
\end{equation}
for some choice for the parameter $M_r$.

The solution of the linear system \eqref{S:one} shows that the coefficients $\{a_n, b_n | n=0,\dots,N \}$ decay quickly and, therefore, we choose the truncation parameter to be $N=16$. Once the coefficients $\{a_n, b_n | n=0,\dots,N \}$ are found, the spectral functions $\rho_{jk}(t)$ can be computed. The function$f(z)$ can be then computed via the transform pair \eqref{E:TP}. We have verified our numerical results with those obtained by Shepherd \cite{Shepherd1937} and Crowdy \cite{Crowdy:2015}. Figure \ref{fig: comparison1} shows the real and imaginary parts of $f(\zeta(\theta))$, $ \theta \in [-\pi/2, 3\pi/2]$ along the boundary of the unit disc computed using our new formulation, as well as Shepherd's \cite{Shepherd1937} and Crowdy's \cite{Crowdy:2015} solutions. We observe that there appear to be discontinuities in the real part of $f$ for $\theta=\pi/2$ and $\theta=3\pi/2$ (also corresponding to $\theta=-\pi/2$), values at which the boundary conditions change type. Discontinuities as such are not surprising given the solution is in $E^1$.

\begin{figure}[H]
\centering
\includegraphics[scale=0.36]{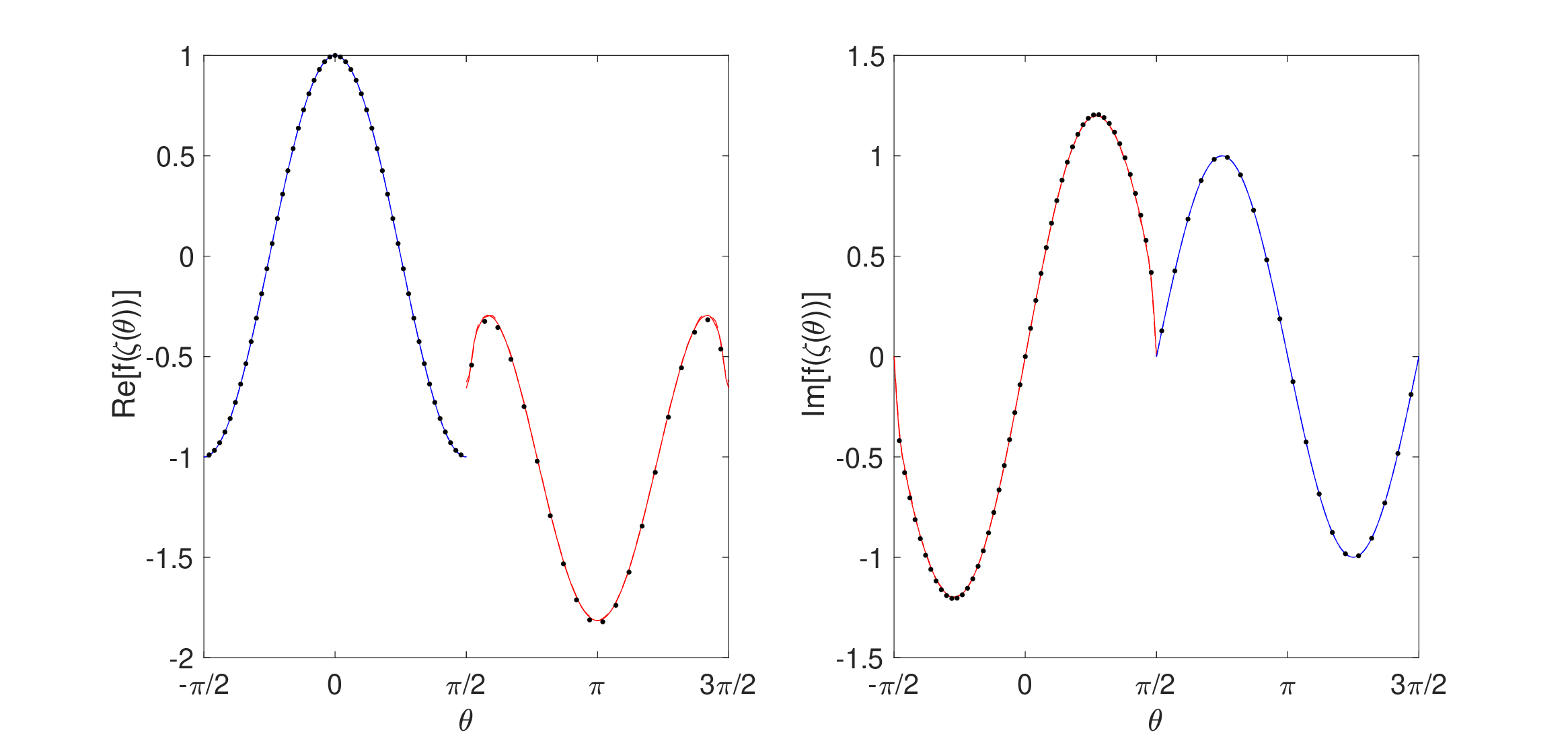}
\caption{Real (left) and imaginary (right) parts of $f(\zeta(\theta))$, $ \theta \in [-\pi/2, 3\pi/2]$ along the boundary of the unit disc, for $m=2$: new formulation for $N=16$ (solid lines), Crowdy's \cite{Crowdy:2015} solution for the same truncation parameter (dashed lines) and Shepherd's \cite{Shepherd1937} solution (dots). The solutions via the new formulation and Crowdy's \cite{Crowdy:2015} solution are indistinguishable.}
\label{fig: comparison1}
\end{figure}

\section{Dual Fourier Series for the Ellipse}

In this section, we present a generalization of the mixed boundary value problem considered by Shepherd \cite{Shepherd1937} on the disc to a problem posed on an elliptical domain $D$, namely   
%\Red{[SGLS: explain that this is a generalization of Shepherd.]} Given the ellipse
\begin{equation}\label{ellipse eqn}
D \coloneqq \left\{z \in \mathbb C, x=\text{Re}[z], y=\text{Im}[z] \quad \Big| \quad \frac{x^2}{a^2}+\frac{y^2}{b^2}<1\right\},
\end{equation}
whose boundary is
\begin{equation}
\partial D=\left\{z \in \mathbb C, x=\text{Re}[z], y=\text{Im}[z] \quad \Big| \quad \frac{x^2}{a^2}+\frac{y^2}{b^2}=1\right\}.
\end{equation}
We emphasize that the new methodology can be applied to solve boundary value problems in any convex domain; here we have chosen to focus on the ellipse. We consider the following incomplete boundary value problem for analytic functions
\begin{equation}\label{E:BVP-ell-good}
\begin{cases}
\overline{\partial}f (z) = 0,\qquad z \in D, \\ \\ %\frac{x^2}{a^2}+\frac{y^2}{b^2}<1, \\ \\
\Real f(\zeta)= \Real\hspace{1mm} \overline{\zeta}^m, \qquad \zeta\in \partial D \text{ and } \Real \zeta>0, \\ \\
\Imag f(\zeta)=\Imag \hspace{1mm} \overline{\zeta}^m , \qquad \zeta\in \partial D \text{ and }  \Real \zeta<0,
\end{cases}
\end{equation}
for a given $m\in\mathbb N$.
%% SIMPLE PARAMETRIZATION OF THE ELLIPSE (but we'll not use it in this section)
%We parametrize the ellipse $D$ as %\Red{[SGLS: $\mathbb E$ is very ugly in integrals; shouldn't this just be $\zeta$?]}
%\begin{equation} 
%\zeta(\theta)=a\cos\theta+\ci b\sin\theta=\frac{1}{2}\Big((a+b)\ee^{\ci \theta}+(a-b)\ee^{-\ci \theta}\Big),\quad \Red{\theta \in (-\pi/2, 3\pi/2]}.
%\end{equation}
%From this, we have
%\begin{equation} 
%\zeta'(\theta) =\frac{\ci}{2}\Big((a+b)\ee^{\ci \theta}-(a-b)\ee^{-\ci \theta}\Big).
%\end{equation}

We parametrise the ellipse using polar coordinates $(\rho,\theta)$, with $0\leq \theta\leq 2\pi$ and
\begin{equation} 
\rho(\theta)=\frac{ab}{\sqrt{(b\cos\theta)^2+(a\sin\theta)^2}}
\end{equation}
%DO WE WANT THIS? HOW TO FIND \rho(\theta):
%Indeed, given a $0\leq \theta\leq 2\pi$, we want to find $\rho>0$ such that
%\begin{equation}
%\frac{\rho^2 \cos^2\theta}{a^2}+\frac{\rho^2\sin^2\theta}{b^2}=1.
%\end{equation}
%Solving for $\rho$ in the equation above gives $\rho(\theta)$ as stated above.
and write       
%We may also parameterize the ellipse with respect to polar coordinates:
\begin{equation}\label{parametrization_polar}
\zeta(\theta) \coloneqq \rho(\theta)\ee^{\ci \theta} \in D.
\end{equation}

%We can write $\overline \zeta^m$ in polar coordinates as
%\begin{equation}
%\overline \zeta^m=[\rho(\theta)]^m \ee^{-\ci m\theta}.
%\end{equation} 

\noindent
The boundary value problem \eqref{E:BVP-ell-good} can be written in terms of polar coordinates as
\begin{equation}\label{E:BVP-ell-good_polar}
\begin{cases}
\overline{\partial}f (z) = 0,\qquad z \in D, \\ \\
\Real f= [\rho(\theta)]^m \cos m\theta, \qquad \theta\in C_1\coloneqq (-\pi/2, \pi/2), \\ \\
\Imag f=- [\rho(\theta)]^m \sin m\theta, \qquad \theta\in C_2 \coloneqq (\pi/2, 3\pi/2).
\end{cases}
\end{equation}
%We extend the given data to $\mathbb E_{a,b}$ as follows:
%seek $f(E_{a,b}(\theta))$, $\overline{\partial}f (z) = 0$ in $E_{a,b}$, taking the following form:

On $C_1$, we write   
\begin{equation}\label{E:apriori-1}
f(\theta,\rho(\theta))=[\rho(\theta)]^m \cos m\theta+\ci\Big(a_0+\sum_{n\geq 1}\Big[a_n\ee^{2\ci n\theta}+\overline{a_n} \ee^{-2 \ci n\theta}\Big]\Big).  
\end{equation}
where coefficients $\{a_n \in \mathbb C | n=0,1,2,\dots \}$ are to be determined.

On $C_2$, we write
\begin{equation}\label{E:apriori-2}
f(\theta,\rho(\theta))= \Big( b_0+\sum_{n\geq 1}\Big[b_n\ee^{2\ci n\theta}+\overline{b_n} \ee^{-2\ci n\theta}\Big] \Big) -\ci [\rho(\theta)]^m \sin m\theta.
\end{equation}
where coefficients $\{b_n \in \mathbb C | n=0,1,2,\dots \}$ are to be determined.

We have the following global relation:
\begin{equation}\label{E:global_ellipse}
\sum_{n=1}^4\rho_{1n}(t)=\sum_{n=1}^4\,\int\limits_{I_n(z)}\!\!\! f(\zeta)\ee^{-\ci t\zeta}\dd\zeta=\int_{\partial D}f(\zeta)\ee^{-\ci t\zeta}\dd\zeta =0, \quad t \in \mathbb{C},
\end{equation}
where the last identity is due to Cauchy theorem applied to $F(\zeta) \coloneqq f(\zeta)\ee^{-\ci t\zeta}$ (which is analytic in $D$ for each fixed $t \in \mathbb{C}$. (Note that we need to strengthen our assumption on $f$ to ensure applicability of Cauchy theorem, e.g. $f\in H^1(D)$ as in the proceeding sections.)
%We wish to write the global relation in terms of the extended data. To do this, it is simplest to write the part of the extended data with the real and imaginary part of $\overline{\zeta}^m$ with respect to the polar parameterization and the other part of the extended data with respect to the standard  parameterization:
%\begin{equation}\label{simple ellipse param}
%%\zeta(\theta)=(a\cos\theta,b\sin \theta).
%\zeta(\theta)=a\cos \theta +\ci b\sin \theta.
%\end{equation}

On substitution of \eqref{E:apriori-1} and \eqref{E:apriori-2} into the global relation \eqref{E:global_ellipse} and evaluation of \eqref{E:global_ellipse} at certain values of $t \in \mathbb{C}$, we obtain a linear system for the unknown coefficients $\{a_n, b_n | n=0,\dots,N \}$ of $f$. The linear system is given by \eqref{S:one}, but now $\A(n,t)$, $\B(n,t)$ and $r(t)$ are defined by
%\Red{
%\begin{equation}\label{linearsystem_ellipse}
%\begin{split}
%&a_0\A(0,t)+\sum_{n\geq 1}\Big(a_n\A(2n,t)+\overline{a_n} \A(-2n,t)\Big) \\
%&~~~~~~+b_0\B(0,t)+\sum_{n\geq 1}\Big(b_n \B(2n,t)+\overline{b_n} \B(-2n,t)\Big)=r(t), \quad t \in \mathbb{C}, \\
%&a_0\overline{\A(0,t)}+\sum_{n\geq 1}\Big(\overline{a_n} \overline{\A(2n,t)}+ a_n\overline{\A(-2n,t)}\Big) \\
%&~~~~~~+b_0\overline{\B(0,t)} +\sum_{n\geq 1}\Big(\overline{b_n} \overline{\B(2n,t)}+ b_n\overline{\B(-2n,t)}\Big)=\overline{r(t)}, \quad t \in \mathbb{C},
%\end{split}
%\end{equation}
%}
%Here $\A(n,t)$, $\B(n,t)$ and $r(t)$ are defined by
\begin{equation}\label{ABellipse}
\begin{split}
\A(n,t)&= \ci \int_{-\pi/2}^{\pi/2}\ee^{-\ci t\zeta(\theta)}\ee^{\ci n \theta} \zeta'(\theta) \dd\theta, \\ \B(n,t)&= \int_{\pi/2}^{3\pi/2}\ee^{-\ci t\zeta(\theta)}\ee^{\ci n \theta} \zeta'(\theta) \dd\theta,
\end{split}
\end{equation}
and
\begin{equation}\label{rellipse}
r(t)=-\int_{-\pi/2}^{\pi/2} [\rho(\theta)]^m \cos (m\theta) \ee^{-\ci t \zeta(\theta)}\zeta'(\theta)\dd\theta
+ \ci \int_{\pi/2}^{3\pi/2} [\rho(\theta)]^m \sin (m\theta)\ee^{-\ci t \zeta(\theta)}\zeta'(\theta)\dd\theta.
\end{equation}
%with
%\begin{equation}
%\zeta'(\theta)=-\frac{1}{2}\frac{ab((a^2-b^2)\sin2\theta)}{((b\cos \theta)^2+(a\sin\theta)^2)^{3/2}}\ee^{\ci\theta}+\ci\frac{ab}{((b\cos\theta)^2+(a\sin\theta)^2)^{1/2}}\ee^{\ci\theta},
%\end{equation}
%obtained by differentiation of \eqref{parametrization_polar}.
Note that if $a=b=1$, i.e. we work with the unit circle, then expressions \eqref{ABellipse}--\eqref{rellipse} become identical to \eqref{ABdisc}--\eqref{rdisc}.

%Note that, if $a=b$, i.e. the ellipse is a circle, then the six terms with the factor of $a-b$ vanish, so we are left with
%\begin{equation}
%\begin{split}
%&\frac{a+b}{2} a_0\A(0,t)+\sum_{n\geq 1}a_n \Big(\frac{a+b}{2}\A(2n,t)\Big)+\sum_{n\geq 1}\overline{a_n}\Big(\frac{a+b}{2}\A(-2n,t)\Big) \\
%&~~~~~~+\frac{a+b}{2}b_0\B(0,t)+\sum_{n\geq 1}b_n \Big(\frac{a+b}{2}\B(2n,t)\Big)+\sum_{n\geq 1}\overline{b_n}\Big(\frac{a+b}{2}\B(-2n,t)\Big)=r(t), \quad t \in \mathbb{C}.
%\end{split}
%\end{equation}
%If we work with the unit circle, then $(a+b)/2=1$, so we get
%\begin{equation}
%\begin{split}
%&a_0\A(0,t) + \sum_{n\geq 1}a_n \A(2n,t)+\sum_{n\geq 1}\overline{a_n}\A(-2n,t) \\
%&~~~~~~+b_0\B(0,t)+\sum_{n\geq 1}b_n \B(2n,t)+\sum_{n\geq 1}\overline{b_n}\B(-2n,t)=r(t),  \quad t \in \mathbb{C},
%\end{split}
%\end{equation}
%which is consistent with the previous analysis for the circle.

The equations \eqref{S:one} hold for any $t \in \mathbb{C}$. To proceed, the sums  \eqref{E:apriori-1} and  \eqref{E:apriori-2} (for $f(\theta,\rho(\theta)$) are truncated to include only terms up to $n=N$ and we formulate a linear system for the unknown coefficients $\{a_n, b_n | n=0,\dots,N \}$. The linear system comprises of conditions \eqref{S:one} evaluated at points $t \in \mathbb{C}$ which are used to form an overdetermined linear system.This is then solved using least squares. Following a similar approach to Colbrook {\em et al.} \cite{ColbrookFlyerFornberg:2018}, points $t \in \mathbb{C}$ are chosen to be the origin and the concentric ellipses
\begin{equation}
\{t \in \mathbb{C} | t=-j/f'(\theta,\rho(\theta)), \theta \in [0,2\pi], j=1,\dots,M_r\},
\label{GRellipsepoints}
\end{equation}
for some choice for the parameter $M_r$.

\begin{figure*}[t!]
    \centering
    \begin{subfigure}[t]{0.5\textwidth}
        \centering
        \[\includegraphics[scale=0.24] {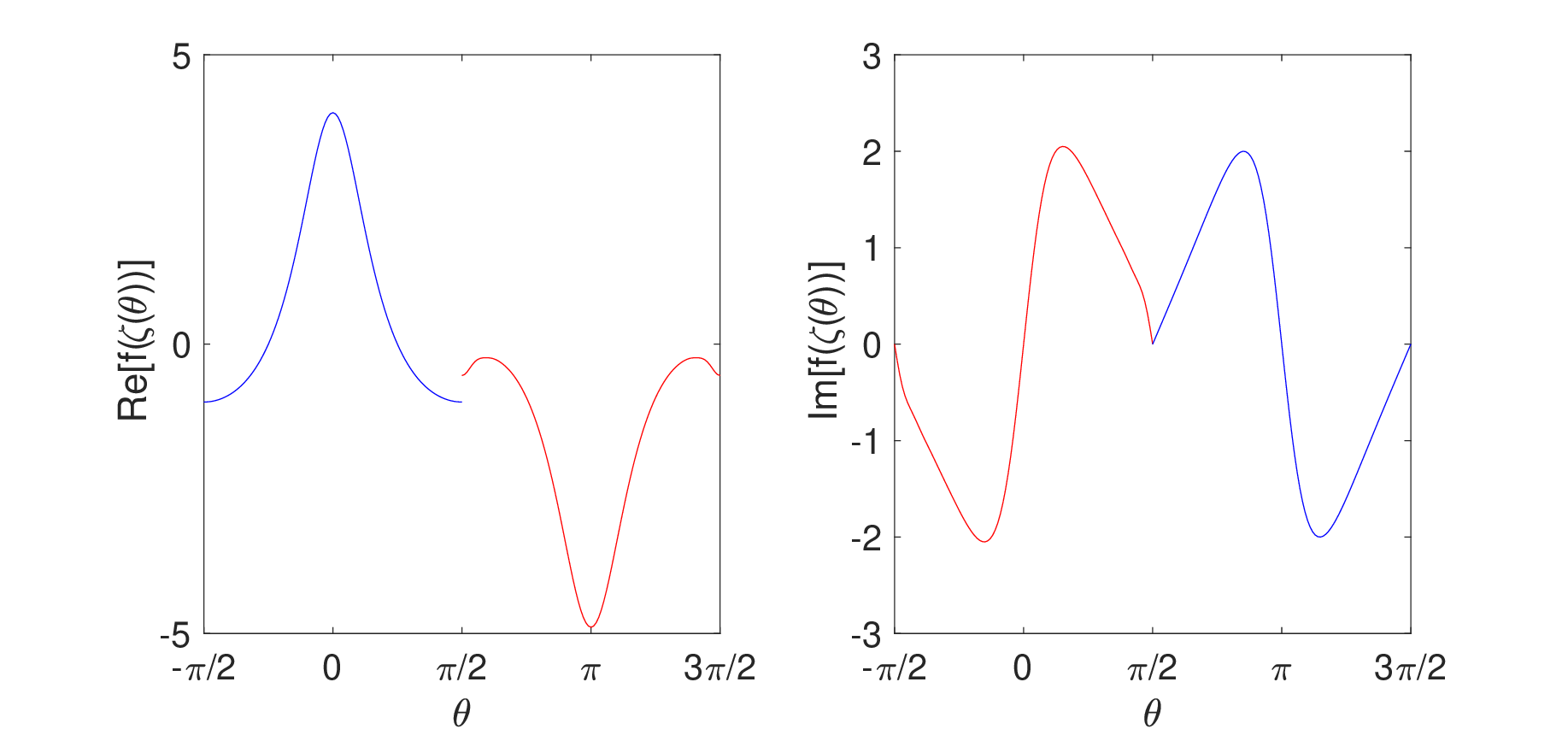} \]
\caption{Parameters: $a=2$, $b=1$.}
\label{Fig:3a}
    \end{subfigure}%
~    
    \begin{subfigure}[t]{0.5\textwidth}
        \centering
  \[\includegraphics[scale=0.24] {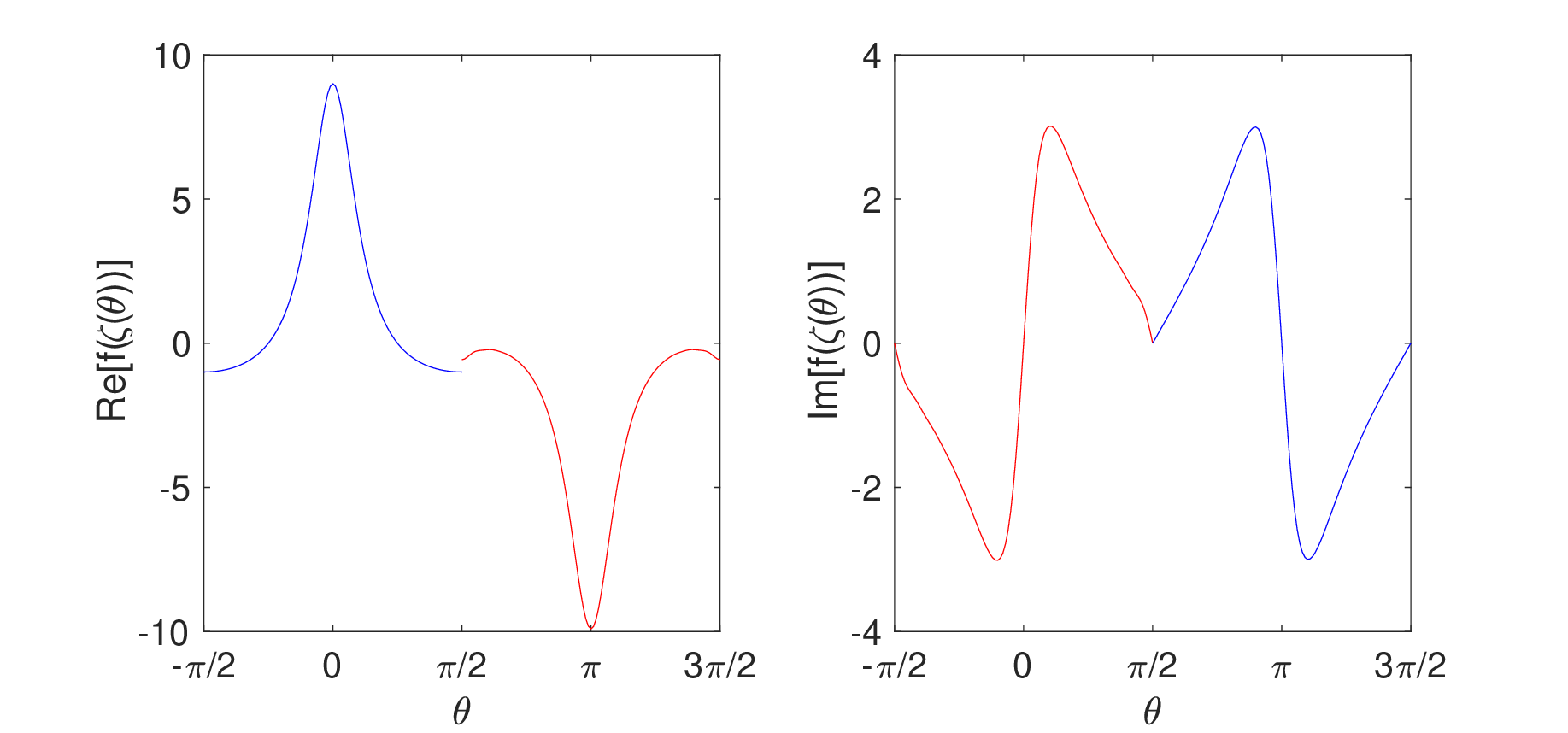} \]
  \caption{Parameters: $a=3$, $b=1$.}
  \label{Fig:3b}
    \end{subfigure}
    \caption{Real and imaginary parts of $f(\zeta(\theta))$, $ \theta \in [-\pi/2, 3\pi/2]$ along the boundary of the ellipse, for $m=2$.}
    \label{fig: ellipse examples}
\end{figure*}

Similarly to the mixed boundary value problem posed on the unit disc presented in the previous section, we found that the coefficients $\{a_n, b_n | n=0,\dots,N \}$ decay quickly and, therefore, we choose the truncation parameter to be $N=16$. Once the coefficients $\{a_n, b_n | n=0,\dots,N \}$ are found, the spectral functions $\rho_{jk}(t)$ and $f(z)$ can be computed via the transform pair \eqref{E:TP}. Figure \ref{fig: ellipse examples} shows the real and imaginary parts of $f(\zeta(\theta))$, $ \theta \in [-\pi/2, 3\pi/2]$ along the boundary of the ellipse, for $m=2$ and different parameter choices for $a$ and $b$. Similarly to the mixed boundary value problem on the unit disc presented in the previous section, we observe that there appear to be discontinuities in the real part of $f$ for $\theta=\pi/2$ and $\theta=3\pi/2$ (also corresponding to $\theta=-\pi/2$), values at which the boundary conditions change type. Discontinuities as such are not surprising given the solution is in $E^1$.

\section{Application in fluid dynamics: A point vortex in the interior of an elliptical boundary} 

%\Red{[SGLS: haven't touched this yet.]}

\noindent
In this section, we present an application of the new formulation to a problem in fluid dynamics, in particular within the framework of a two-dimensional, inviscid, incompressible and irrotational (except of point vortices) steady flow. We consider a point vortex (namely, a vortex with infinite vorticity concentrated at a point) in the interior of a boundary with an elliptical shape and the aim is to find the resulting fluid flow satisfying the imposed impermeability boundary condition.

\subsection{Problem formulation}

We consider a point vortex at point $\zeta_0$ in the interior of the ellipse \eqref{ellipse eqn}. A schematic of the configuration is shown in Figure \ref{fig: point vortex}. To begin, we introduce a complex potential function $h(z)$ and write
%\begin{equation}
%h(\zeta)=\chi(x,y)+ \ci w(x,y),
%\end{equation}
%where $\chi(x,y)$ is the harmonic conjugate to $w(x,y)$. To begin, we write
\begin{equation}
h(\zeta)=f_s(\zeta)+f(\zeta), \label{expansionellipse}
\end{equation}
where
\begin{equation}
f_s(\zeta) = \frac{\Gamma}{2\pi \ci} \log(\zeta-\zeta_0),
\end{equation}
where $\Gamma \in \mathbb{R}$ is the circulation. The function $f(\zeta)$ is analytic in the ellipse and will be found using the transform method \eqref{E:repr-f}.

\begin{figure}[H]
\centering
\includegraphics[scale=0.8]{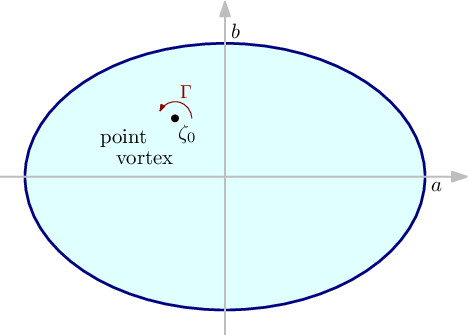}
\caption{Schematic of the configuration: A point vortex at $\zeta_0$ in the interior of an elliptical boundary.}
\label{fig: point vortex}
\end{figure}

\subsection{Solution scheme}

\subsubsection{Boundary condition}
We impose an impermeability condition on the boundary of the ellipse:
\begin{equation}\label{impermeability}
{\bf u} \cdot {\bf n}=0,
\end{equation}
where ${\bf u}=(u,v)$ is the two-dimensional velocity field and ${\bf n}$ denotes the direction outward normal to the ellipse. Condition \eqref{impermeability} can be expressed in terms of the complex potential as
\begin{equation}\label{BCvortex}
\text{Im}[h(\zeta)]=0.
\end{equation}
%This follows from the Cauchy--Riemann equations, from which one may deduce that $w(x,y)$ must be constant along the boundary. Setting this constant to be zero accounts for the arbitrary additive real constant up to which $w(x,y)$ is otherwise specified.
Substitution of \eqref{expansionellipse} into \eqref{BCvortex} gives
\begin{equation}\label{BCvortex2}
\text{Im}[f(\zeta)] = - \text{Im}[f_s(\zeta)].
\end{equation}

%\begin{equation}
%\psi=\text{Im}[w(\zeta)]=0. \label{BC}
%\end{equation}
%Substitution of \eqref{expansion} into \eqref{BC} gives
%\begin{equation}
%\text{Im}[f(\zeta)]= -\text{Im}[f_s(\zeta)]. \label{BC2}
%\end{equation}

\subsubsection{Function representation}

We represent $f(\zeta)$ on the boundary of the ellipse using a Fourier expansion:
\begin{equation}\label{function rep1}
f(\zeta(\theta)) = \Big(\sum_{n=0}^{\infty} a_n \ee^{\ci n \theta} + \sum_{n=0}^{\infty} \overline{a_n} \ee^{-\ci n \theta} \Big) - \ci ~\text{Im}[f_s(\zeta)], \quad \text{for } \theta \in [0,2\pi],
\end{equation}
where coefficients $\{a_n \in \mathbb C | n=1,2,\dots \}$ are to be determined. The parametrization of the ellipse is given by
\begin{equation}
\zeta(\theta)=a \cos\theta +\ci b \sin\theta, \quad \theta \in [0,2\pi].
\end{equation}

Note that, since the imposed boundary condition \eqref{impermeability} is the same along the entire elliptical boundary (in contrast to the previous problems where different conditions where imposed on different sections of the boundary), we have used a single representation \eqref{function rep1} for $f(\zeta)$ to proceed instead of different representations on different sections of the boundary.

\subsubsection{Spectral analysis}

On substitution of \eqref{function rep1} into the global relation \eqref{E:global_ellipse}, we obtain (after some algebra and rearrangement) a linear system for the unknown coefficients $\{a_n | n=1,\dots,N\}$. The infinite sums are truncated to include terms up to $n=N$. The linear system is given by

\begin{equation}
\sum_{n=0}^{N} \Big(a_n P(n,t) + \overline{a_n} P(-n,t)\Big)=R(t), \quad \text{for } t \in \mathbb{C},
\label{linearsystempointsingularity}
\end{equation}
where
\begin{equation}
P(n,t)=\int_{0}^{2\pi} {\ee^{\ci n\theta}\ee^{-\ci t \zeta(\theta)} \zeta'(\theta) \dd \theta}
\end{equation}
and
\begin{equation}
R(t) = \ci \int_{0}^{2\pi} {\text{Im}[f_s(\zeta(\theta))] \ee^{-\ci t \zeta(\theta)} \zeta'(\theta) \dd \theta}.
\end{equation}
The linear system comprises of \eqref{linearsystempointsingularity} and its conjugate evaluated at points $t \in \mathbb{C}$ which are used to form an overdetermined linear system. The choice of points $t \in \mathbb{C}$ is the same as in the previous section and given by \eqref{GRellipsepoints}. We found that the coefficients $\{a_n | n=1,\dots,N \}$ decay quickly and, therefore, we choose the truncation parameter to be $N=16$. Once the coefficients are found, the spectral functions and $f(z)$ can be computed via the transform pair \eqref{E:TP}.

Our results were checked against an exact solution found using conformal mapping techniques; introduce the conformal mapping $\Phi$ from the ellipse $D$ in the $\zeta$-plane to the unit disc $\mathbb{D}$ in the $w$-plane to be given by (Schwarz \cite{Schwarz:1869}):
\begin{equation}
\begin{split}
w=\Phi(\zeta) &= \sqrt{k} ~\text{sn} \left(\frac{2K}{\pi} \sin^{-1}(\zeta),m \right),
\end{split}
\end{equation}
where $\text{sn}(\cdot, \cdot)$ is the Jacobi sn function and definition of parameters $k,K$ and $m$ is given in \cite{Schwarz:1869,Szego:1950}. The complex potential function for a point vortex in the ellipse is given by
\begin{equation}\label{exactpotential}
H(w)=\frac{\Gamma}{2\pi \ci} \log\left(\frac{w-w_0}{w-1/\overline{w_0}}\right), \qquad \text{with} \quad w_0=\Phi(\zeta_0).
\end{equation}
We have compared our results to \eqref{exactpotential} for interior points and the error at interior points was of the order $\mathcal{O}(10^{-4})$. To illustrate our results, Figure \ref{fig: contour} shows the streamline pattern for a point vortex of strength $\Gamma=1$ at point $\zeta_0$ in the interior of an ellipse with parameters $a$ and $b$.

\begin{figure}[H]
%\centering
\hspace*{-1.7cm}
\includegraphics[scale=0.34]{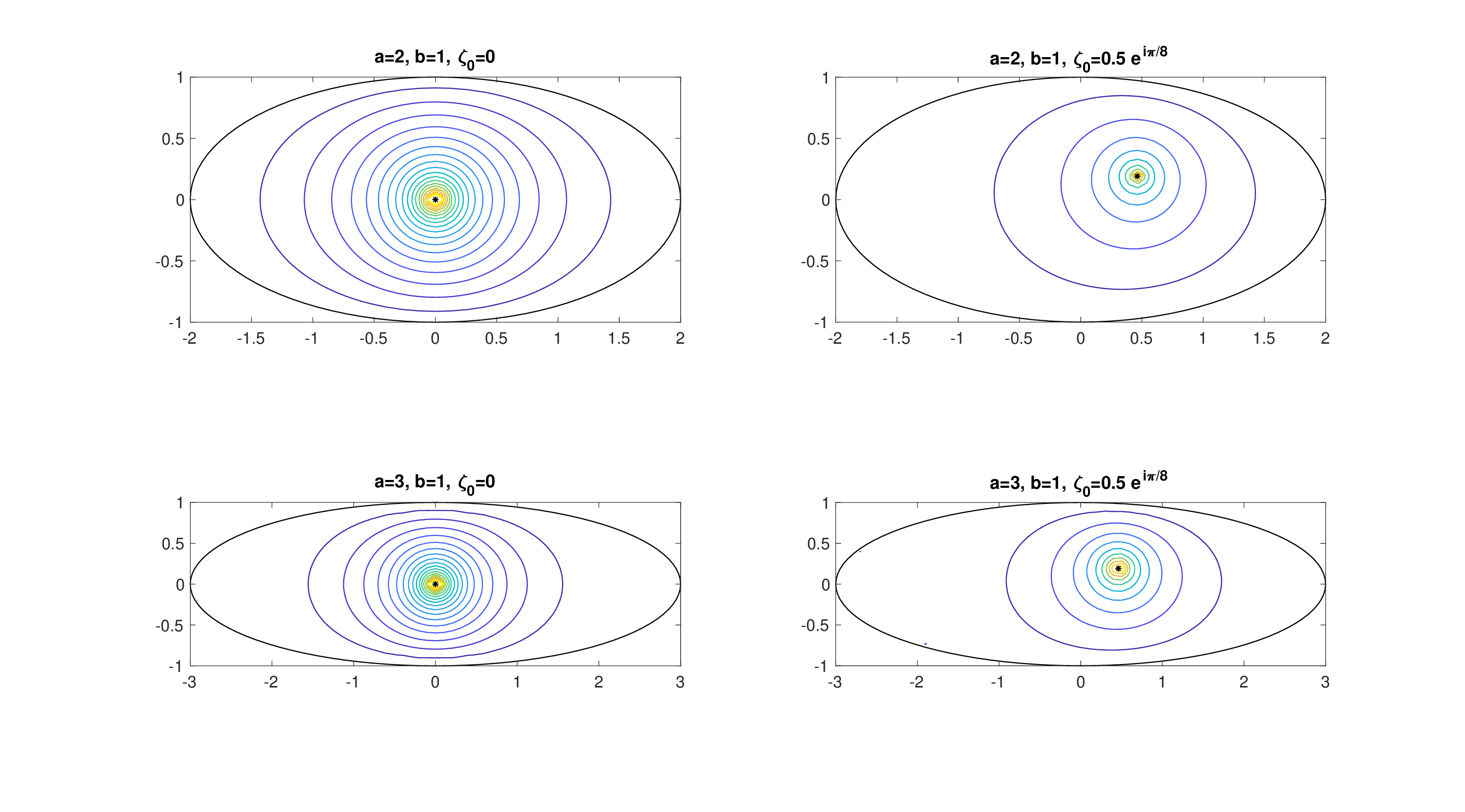}
\vspace{-1cm}
\caption{Streamlines for a point vortex of strength $\Gamma=1$ at point $\zeta_0$ in the interior of an ellipse with parameters $a$ and $b$.}
\label{fig: contour}
\end{figure}

\section{The complex Helmholtz equation}

In this section, we discuss how to obtain a transform pair for the complex Helmholtz equation in bounded convex planar domains.

Hauge \& Crowdy \cite{HaugeCrowdy:2021} presented a transform method for the complex Helmholtz equation in polygonal domains using the theory of Bessels functions and Greens second identity.
Given the geometric results of Lemma {\ref{L:3}}, here we develop a transform pair for the complex Helmholtz equation given by 
\begin{equation}
\Delta \phi -4\sigma \phi=0,
\end{equation}
where $\Delta=\nabla^2=\partial^2/\partial x^2+\partial^2/\partial y^2$ is the Laplacian operator and $0<\Arg(\sigma)<\pi$ that mirrors the transform pair for the $\overline{\partial}$ equation. Two results were pivotal in developing the transform pair for the $\overline{\partial}$ equation: the Cauchy integral formula and the desingularized Cauchy kernel. While a new integral representation for functions that satisfy the complex Helmholtz equation is required, Corollary \ref{C:argument} is still paramount in the desingularization of the new integral kernel.
           
\subsection{An integral representation for solutions of the Helmholtz equation}

The classical Green's second identity:
\begin{equation}
\iint_D \Big[\phi \Delta \psi - \psi\Delta\phi \Big] \dd A = \int_{\partial D} \phi\frac{\partial \psi}{\partial \bold{n}}-\psi\frac{\partial\psi}{\partial \bold{n}} \dd s,
\end{equation}
where  $D$ is a bounded domain with piecewise $C^2$ boundary, $\dd A$ is the Lebesgue measure and $\dd s$ is the surface area measure, holds in $\mathbb{R}^n$ for twice continuously differentiable functions $\phi$ and $\psi$ that extend continuously to the boundary of $D$. If one restricts to $\mathbb{C}$ and uses the the identities
\begin{equation}
\frac{\partial^2 }{\partial z\partial \overline z}=4\Delta, \qquad \dd=\partial +\overline \partial, \qquad \dd A=\dd x \dd y=-2\ci \dd z \dd \overline z
\end{equation}
and Stokes' theorem, we can write Green's second identity in complex notation
\begin{equation}
\iint_D\Big[ \phi \Delta \psi-\psi\Delta\phi\Big] \dd A =2 \ci \int_{\partial D} \frac{\partial \phi}{\partial z}\psi \dd z+\phi\frac{\partial\psi}{\partial \overline z} \dd \overline z.
\end{equation}           
The (free space) Green's Function $\mathcal G(z,z_0)$ for the complex Helmholtz equation is defined as
\begin{equation}
\Delta \mathcal G(z,z_0)-4\sigma \mathcal G(z,z_0)=\delta(z-z_0).
\end{equation}           
If we let $G=\mathcal G$ and let $\phi$ be the solution to the complex Helmholtz equation (i.e. $\Delta \phi=4\sigma \phi$), then we get
\begin{equation}
\begin{split}
\iint_D \Big[ \phi(z) \Delta G(\zeta,z)-G(\zeta,z)\Delta\phi\Big] \dd A&=\iint_D \Big[ \phi(z) \Delta G(\zeta,z)-G(\zeta,z)4\sigma \phi(\zeta) \Big] \dd A \\
&=\iint_D \phi(\zeta)\Big[\Delta G(\zeta,z)-G(\zeta,z)4\sigma\Big] \dd A \\
&=\iint_D \phi(\zeta)\delta(\zeta-z) \dd A \\
&=\phi(z).
\end{split}
\end{equation}
Hence we obtain
\begin{equation}
\phi(z)=2 \ci \int_{\partial D}\Big[\phi(\zeta)\frac{\partial G(\zeta,z)}{\partial \overline \zeta} \dd \overline \zeta+G(\zeta,z)\frac{\partial \phi(\zeta)}{\partial \zeta} \dd \zeta \Big]. \label{(7.7)}
\end{equation} 
Note that the formula above gives us a way to recover interior values of $\phi$ by only knowing the boundary values of $\phi$. Formula \eqref{(7.7)} plays for the complex Helmholtz equation the analogous role as the Cauchy integral formula for the $\overline{\partial}$ equation. In order to compute this integral, one needs an explicit formula for $\mathcal G$.

\subsection{The desingularization of the Green's function}

The Green's function $\mathcal G$ for the complex Helmholtz equation can be expressed in terms of the order-zero modified Bessel function $K_0$. The function $K_0$ is axissymmetric, singular as $z\rightarrow z_0$ and decays as $z\rightarrow \infty$. For $\sigma$ as above, we have
\begin{equation}
\mathcal G(z,z_0)=-\frac{1}{2\pi}K_0(2\sqrt \sigma |z-z_0|).
\end{equation}

Hauge \& Crowdy \cite{HaugeCrowdy:2021} gave an integral representation for $\mathcal G$.
\begin{equation}
K_0(z)=\frac{1}{2}\int_{L_w} \ee^{-(z/2)(\alpha+1/\alpha)} \dd \alpha/\alpha, \qquad -\frac{\pi}{2}<\arg(z)-\theta<\frac{\pi}{2},
\end{equation}
where $\theta$ is the geometrical angle corresponding to the contour $L_W$ in the complex plane joining the two essential singularities of the integrand at 0 and $\infty$. The contour must be chosen so that the integrand decays along the contour as it approaches 0 and $\infty$. If we let $\theta_0$ be the angle that the contour makes with the at the origin, we need $-\pi/2<\arg(z)-\theta_0<\pi/2$. If the contour approaches $\infty$ at an angle of $\theta_\infty$, then we need $-\frac{\pi}{2}<\arg(z)+\theta_\infty<\frac{\pi}{2}$. If we choose $\theta_0=\theta_\infty$, and set $\theta=\theta_0$, the two inequalities above produce the same $\pi$-length range of possible arguments for $\arg u$. Letting $L_\alpha$ be a contour that approaches $0$ and $\infty$ at an angle of $\theta$, we have that 
\begin{equation}
\mathcal G(z,z_0)=-\frac{1}{4\pi}\int_{L_\alpha} \ee^{-\sqrt\sigma |z-z_0|(\alpha+1/\alpha)}\frac{\dd \alpha}{\alpha}.
\end{equation}
Using the change of variables 
\begin{equation}
- \ci t \ee^{\ci \phi}=\sqrt\sigma \alpha, \qquad \phi=\arg(z-\zeta),
\end{equation}
we define
\begin{equation}
G(\zeta,z) \coloneqq \mathcal G(\zeta,z)=-\frac{1}{4\pi}\int_{L^{(\sigma)}}\ee^{\ci t(z-\zeta)- \ci \sigma\overline{(z-\zeta)}/t}\frac{\dd t}{t},
\end{equation}
where $L^{(\sigma)}$ is the contour $L_\alpha$ after the change of variables. Suppose now we pick a contour $L^{(\sigma)}_0$ so that our integral representation is valid for $0<\arg(z-\zeta)<\pi$. The exact contour does not matter, as long as we have $\theta_0=\arg(\sigma), \theta_\infty=0.$ So, we have
 \begin{equation}
G(\zeta,z)=-\frac{1}{4\pi}\int_{L^{(\sigma)}_0} \ee^{\ci t(z-\zeta)-\ci \sigma\overline{(z-\zeta)}/t}\frac{\dd t}{t}.
 \end{equation}
Now suppose we wish to write the Green's function for a different $\pi$-width range of arguments given by $\chi<\arg(z_0-z)<\chi+\pi$. Let $L^{(\sigma)}_\chi$ be such a contour so that integral representation of the Green's function is valid for such $z$ and $z_0$. Again, the exact contour does not matter, as long as $\theta_0=\arg\sigma -\chi$, $\theta\infty=-\chi$.
Then for this $\pi$-width range of arguments, we have
 \begin{equation}
G(\zeta,z)=-\frac{1}{4\pi}\int_{L^{(\sigma)}_\chi} \ee^{\ci t(z-\zeta)-\ci \sigma\overline{(z-\zeta)}/t}\frac{\dd t}{t}.
\end{equation}
Note that we may chose $L^{(\sigma)}_\chi$ to be the rotation through an angle of $-\chi$ of the contour $L^{(\sigma)}_0$.

\subsection{Transform pair}

Let $\Omega$ be a bounded convex domain in the complex plane. Let $z\in \Omega$. Lemma \ref{L:2} gives a partition $\{I_1,\dots,I_4\}$ of the boundary of $\Omega$. Using the notation established in section 2, we have the following definition.

\begin{defi}\label{D:spectral-jk}
Let $\Omega$ be a bounded convex domain in $\mathbb C$  and let $\phi$ be a $C^2$ function in some neighborhood of $\Omega$ and a solution to the complex Helmholtz equation. The \textbf{spectral functions $\{\rho_{j}(t) : t\in\mathbb C\setminus\{0\}\}_{1\leq j\leq 4}$} are defined as
\begin{equation}\label{D:spectral-jk}
\rho_j(t)=\int_{I_j}\ee^{-\ci t\zeta + \ci \sigma \overline \zeta/t}\Big[\phi(\zeta)(\ci \sigma/t) \dd \overline \zeta+\frac{\partial \phi(\zeta)}{\partial z}\dd \zeta\Big], \qquad t\in\mathbb C\setminus\{ 0\}, \quad j=1,\ldots, 4.
\end{equation}
\end{defi}
 
\begin{lem}\label{L:5}
With same notations and assumptions as above, we have
\begin{equation}\label{E:global-relation}
\sum_{j=1}^4\rho_{j}(t)=0\quad t\in\mathbb C\setminus\{0\},\quad j=1,\ldots, 4.
\end{equation}
\end{lem}
 We call \eqref{E:global-relation} the {\bf global relation}. Proof is presented in the Appendix \ref{appendixA}.

\begin{thm}\label{TransformCHE}%[H-L-LS-L, 2023]
Let $\Omega$ be a bounded convex domain in $\mathbb C$  and let $\phi$ be a solution to the complex Helmholtz equation that is $C^2$ in some neighborhood of $\Omega$. Then, with same notations as above, we have that
\begin{equation}\label{E:repr-fh}
\displaystyle{\phi(z)=\frac{1}{2\pi \ci}\sum_{j=1}^4\int_{L^{(\sigma)}_{\beta_j}} \ee^{\ci t z- \ci\sigma \overline z/t}\rho_j(t)\frac{\dd t}{t}}.
\end{equation}
\end{thm}
We refer to \eqref{D:spectral-jk} and {\eqref{E:repr-fh} as a {\bf transform pair}.} Proof is presented in the Appendix \ref{appendixB}.

\section{Discussion}

We have presented a new transform pair for Laplace's equation and for the complex Helmholtz equation in bounded convex domains. Our work extends the work of Fokas \& Kapaev \cite{FokasKapaev2003} for convex polygons, to arbitrary convex domains. The method was built upon Crowdy's \cite{Crowdy:2015} construction for Laplace's equation in circular domains. We analysed mixed boundary value problems in circular and elliptical domains with different boundary conditions imposed on different sections of the boundary. Our results were verified against the solutions presented by Shepherd \cite{Shepherd1937} and Crowdy \cite{Crowdy:2015} for the mixed boundary value problem on the unit disc.

The advantage of this study is that the new transform pairs and the approach presented in this paper can be algorithmically adapted to solve harmonic problems in bounded convex domains with mixed boundary conditions. We emphasise that, even though one can use conformal mapping techniques to analyse such problems, it becomes tricky or even impossible to find mappings for problems which involve mixed boundary conditions. We also note that this new approach can be also used to analyse boundary value problems for the biharmonic equation (which will involve solving for two analytic functions; Langlois \cite{Langlois1964}), as well as the complex Helmholtz equation as discussed in \S 7. %Similar steps were followed in \S 7 to obtain the transform pair and associated global relation for such problems.

For future work, we aim to adapt the method for solving mixed boundary value problems in multiply connected domains. To solve such problems, one will need to construct transform pairs for non-convex domains, as well as to investigate the effect of having boundary components of different type (polygonal, circular or any smooth boundaries).

\section*{Acknowledgments}
The authors would like to thank the Isaac Newton Institute for Mathematical Sciences, Cambridge, for support and hospitality during the programme {\it Complex analysis: techniques, applications and computations} where part of the work on this paper was initiated. The authors acknowledge helpful discussions with N. Chalmoukis and M. Colbrook.

\section*{Funding}
This work was partly supported by EPSRC grant no EP/R014604/1 and NSF-DMS-1901978 (Lanzani).

\appendix

\section{Additions on the complex Helmholtz equation}

\subsection{Proof of Lemma \ref{L:5} \label{appendixA}}

 \begin{proof}
 First, set
 \begin{equation}
\psi(\zeta) \coloneqq \ee^{- \ci t\zeta+\ci \sigma \overline \zeta/t}.
\end{equation}
 A couple computations show 
\begin{equation}
\Delta \psi=4\sigma \psi \quad \text{and} \quad \frac{\partial }{\partial \overline \zeta}\psi=\frac{\ci \sigma}{t}\psi.
\end{equation}
Indeed, setting $\zeta=x+\ci y$, we have
\begin{equation}
\begin{split}
\Delta \psi &= \frac{\partial^2}{\partial x^2}\ee^{-\ci t(x+\ci y)+\ci\sigma(x- \ci y)/t}+\frac{\partial^2}{\partial y^2}\ee^{-\ci t(x+\ci y)+\ci \sigma(x-\ci y)/t} \\
&= \frac{\partial}{\partial x}(-\ci t+\ci\sigma/t)\ee^{-\ci t(x+\ci y)+\ci \sigma(x-\ci y)/t}+\frac{\partial}{\partial y}(t+\sigma/t)\ee^{-\ci t(x+\ci y)+\ci \sigma(x-\ci y)/t} \\
&=(-\ci t+\ci \sigma/t)^2\ee^{-\ci t(x+\ci y)+\ci \sigma(x-\ci y)/t}+(t+\sigma/t)^2\ee^{-\ci t(x+\ci y)+\ci \sigma(x- \ci y)/t}\\&=((-\ci t+\ci \sigma/t)^2+(t+\sigma/t)^2)\psi \\
&=4\sigma \psi
\end{split}
\end{equation}
and
\begin{equation}
\frac{\partial }{\partial \overline \zeta}\psi=\frac{\partial }{\partial \overline \zeta}\ee^{-\ci t\zeta+\ci \sigma \overline \zeta/t}=\big[\frac{\partial }{\partial \overline \zeta}(-\ci t\zeta+\ci \sigma\overline \zeta /t)\big]\ee^{-\ci t\zeta+\ci \sigma \overline \zeta/t}=\frac{\ci \sigma}{t} \psi.
\end{equation}
Using these two computations above, the complex version of Green's second identity,  and $\Delta \phi=4\sigma \phi$ ($\phi$ is the solution to the Helmholtz equation), we have that
\begin{equation}
\begin{split}
\sum_{j=1}^4\rho_j(t)&=\int_{\partial \Omega } \psi(\zeta)\Big[\phi(\zeta)(\ci \sigma/t)\dd\overline \zeta+\frac{\partial \phi(\zeta)}{\partial \zeta} \dd\zeta\Big]\\
&=\int_{\partial \Omega } \Big(\frac{\partial}{\partial\overline \zeta}\psi(\zeta)\Big)\phi(\zeta) \dd\overline \zeta+\frac{\partial \phi(\zeta)}{\partial \zeta}\psi(\zeta) \dd\zeta\Big] \\
&=\frac{1}{2 \ci}\iint_\Omega\Big( \phi \Delta \psi-\psi\Delta\phi\Big) \dd A \\
&=\frac{1}{2 \ci}\iint_\Omega\Big( \phi 4\sigma \psi-\psi4\sigma\phi\Big) \dd A=0.
\end{split}
\end{equation} 
 \end{proof}

\subsection{Proof of Theorem \ref{TransformCHE} \label{appendixB}}

\begin{proof}
Given a point $z\in \Omega$, we inscribe a trapezoid $T$ as in Lemma \ref{L:2} that contains $z$. The vertices of $T$ partition the boundary of $T$ into 4 arcs, $I_j$, $1\leq j\leq 4$. For each $j$, we have the conformal affine map
\begin{equation}
\Psi_j(w)=\ee^{-\ci \beta_j}(w-\alpha).
\end{equation}
Corollary \ref{C:argument} states that for $(z,\zeta)\in\text{int}(T)\times I_j$
\begin{equation}
0<\arg(\Psi_j(z)-\Psi_j(\zeta))<\pi.
\end{equation}
A simple calculation shows
\begin{equation}
\ee^{-\ci \beta_j}(z-\zeta)=\Psi_j(z)-\Psi_j(\zeta),
\end{equation}
so it follows that
\begin{equation}
\beta_j<\arg(z-\zeta)<\beta_j+\pi,
\end{equation}
where $\beta_j\in[0,2\pi)$, $1\leq j\leq 4$ are angles of rotation determined by the trapezoid $T$ and $(z,\zeta)\in\text{int}(T)\times I_j$. We let $L^{(\sigma)}_0$ be a contour so that our integral representation for $G$ is valid for $0<\Arg(z-\zeta)<\pi$. For each $\beta_j$, $1\leq j\leq 4$, we may rotate $L^{(\sigma)}_0$ by $-\beta_j$ to get a contour $L^{(\sigma)}_{\beta_j}$ that is valid for $\beta_j<\arg(z-\zeta)<\beta_j+\pi$. Thus we get the following integral representation for $G$:
\begin{equation}
G(\zeta,z)=-\frac{1}{4\pi}\int_{L^{(\sigma)}_{\beta_j}} \ee^{\ci t(z-\zeta)-\ci \sigma\overline{(z-\zeta)}/t}\frac{\dd t}{t}
\end{equation}
that is valid when
\begin{equation}
\beta_j<\arg(z-\zeta)<\beta_j+\pi, \qquad 0\leq\Arg(\sigma)<\pi.
\end{equation}

If we let $\phi$ be the solution to the complex Helmholtz equation, then we have the following integral representation for $\phi$ that follows from Green's second identity:
\begin{equation}
\phi(z)=2 \ci \int_{\partial \Omega}\Big[\phi(\zeta)\frac{\partial G(\zeta,z)}{\partial \overline \zeta} \dd\overline \zeta+ G(\zeta,z)\frac{\partial \phi(\zeta)}{\partial \zeta} \dd\zeta\Big].
\end{equation}
A quick calculation shows that
\begin{equation}
\frac{\partial G(\zeta,z)}{\partial \overline \zeta}=-\frac{1}{4\pi}\int_{L^{(\sigma)}_{\beta_j}}\ee^{\ci t(z-\zeta)-\ci\sigma\overline{(z-\zeta)}/t}\frac{\ci\sigma}{t}\frac{\dd t}{t}.
\end{equation}
Now that we have a valid integral representation for $G$ and $\partial G/\partial \bar\zeta$ for each $I_j$, we can use Fubini's theorem  and see that
\begin{equation}
\begin{split}
\phi(z)&=2 \ci \sum_{j=1}^k \int_{I_j}\Big[\phi(\zeta)\frac{\partial G(\zeta,z)}{\partial \overline \zeta} \dd\overline \zeta+ G(\zeta,z)\frac{\partial \phi(\zeta)}{\partial \zeta}d\zeta\Big] \\
&=\frac{1}{2\pi \ci}\sum_{j=1}^4\int_{I_j}\int_{L_{\beta_j}^{(\sigma)}} \ee^{\ci t(z-\zeta)- \ci \sigma\overline{(z-\zeta)}/t}\big[\phi(\zeta)(\ci \sigma/t)\dd\overline \zeta+\frac{\partial\phi(\zeta)}{\partial \zeta}\dd\zeta\big]\frac{\dd t}{t} \\
&=\frac{1}{2\pi \ci}\sum_{j=1}^4\int_{L^{(\sigma)}_{\beta_j}} \ee^{\ci tz-\ci\sigma \overline z/t}\rho_j(t)\frac{\dd t}{t}.
\end{split}
\end{equation}
\end{proof}

\bibliography{Transformpair}
\bibliographystyle{plain}

\end{document}